\documentclass[12pt,draft]{article} 
\usepackage{amsmath,amssymb,amsthm,amsfonts, indentfirst, empheq}
\usepackage{enumerate,color,bm,graphicx,here}
\usepackage{multicol}
\usepackage{mathrsfs}
\makeatletter
\def\@cite#1#2{[{{\bfseries #1}\if@tempswa , #2\fi}]}
\renewcommand{\section}{%
\@startsection{section}{1}{\z@}
{0.5truecm plus -1ex minus -.2ex}%
{1.0ex plus .2ex}{\bfseries\large}}
\def\@seccntformat#1{\csname the#1\endcsname.\ }
\makeatother

\numberwithin{equation}{section} 
\theoremstyle{theorem}
\newtheorem{thm}{Theorem}[section]

\newtheorem{lem}[thm]{Lemma}

\theoremstyle{definition}
\newtheorem{df}[thm]{Definition}
\newtheorem{remark}{Remark}[section]

\newcommand{\ep}{\varepsilon}
\newcommand{\pa}{\partial}
\newcommand{\R}{\mathbb{R}}

\begin{document}
\footnote[0]
    {2020{\it Mathematics Subject Classification}\/. 
    Primary: 35K65; Secondary: 35A01, 35A02, 92C17.
    }
    
\begin{center}
    \Large{{\bf Existence and uniqueness of 
                    global weak solutions \\
                    to degenerate volume-filling chemotaxis systems \\
                    with source terms}}
\end{center}
\vspace{5pt}
\begin{center}
    Osuke Shibata%
   \footnote[0]{
    E-mail: 
    {\tt 1125522@ed.tus.ac.jp}
    }\\
    \vspace{12pt}
    Department of Mathematics, 
    Tokyo University of Science\\
    1-3, Kagurazaka, Shinjuku-ku, 
    Tokyo 162-8601, Japan\\
    \vspace{2pt}
\end{center}
\begin{center}    
    \small \today
\end{center}

\vspace{2pt}
\newenvironment{summary}
{\vspace{.5\baselineskip}\begin{list}{}{%
     \setlength{\baselineskip}{0.85\baselineskip}
     \setlength{\topsep}{0pt}
     \setlength{\leftmargin}{12mm}
     \setlength{\rightmargin}{12mm}
     \setlength{\listparindent}{0mm}
     \setlength{\itemindent}{\listparindent}
     \setlength{\parsep}{0pt}
     \item\relax}}{\end{list}\vspace{.5\baselineskip}}
\begin{summary}
{\footnotesize {\bf Abstract.}
This paper is concerned with 
a no-flux initial-boundary value problem 
for the 
degenerate
volume-filling chemotaxis system
with source terms,
\begin{align*}
    \begin{cases}
        u_t = \nabla \cdot (D(u,v) \nabla u - h(u,v) \nabla v) + f(x, u, v),
        & x\in \Omega, \  t>0,
        \\
        v_t = \Delta v + g(u,v),
        & x\in \Omega, \  t>0
    \end{cases}
\end{align*}
in a smoothly bounded domain $\Omega \subset \mathbb{R}^N$ 
$(N \in \mathbb{N})$.
It is shown that 
when 
$D$, $h$, $f$ and $g$
satisfy suitable assumptions involving 
$D(1,\cdot)=h(0,\cdot)=h(1,\cdot)=f(\cdot,0,\cdot) = f(\cdot,1,\cdot) = 0$, 
for nonnegative initial data $u_0$ and $v_0$ 
with $u_0\le 1$,
there exists a global weak solution $(u, v)$ 
with $u\le 1$.
In addition, 
uniqueness of global weak solutions is established
when $D(r,s) = D(r)$
for all $r\in[0,1]$ and $s\in[0,\infty)$, 
and $D$, $h$, $f$, $g$ and $v_0$ are supposed to satisfy additional conditions.} 
\end{summary}

\vspace{10pt}

\newpage

\section{Introduction}\label{Sec:Intro}
{\it Chemotaxis} is a phenomenon in which organisms direct their movement according to
the presence of chemicals. 
In \cite{KS-1917-JTB} and \cite{K-1980-LNinB}
(see also \cite{BBTW-2015-MMMA}),
Keller and Segel introduced the chemotaxis model 
of the form
\begin{align}\label{Sys:KS} 
    \begin{cases}
        n_t = \nabla \cdot (D_n(n,S) \nabla n 
                         - \chi(n,S)n \nabla S) + H(n,S),
        \\
        S_t = D_S(n,S) \Delta S + K(n,S),
    \end{cases}
\end{align}
where $n$ and $S$ represent the organism density
and
the chemoattractant concentration, respectively, and where  
$H$ and $K$ describe sources related to interactions.

When $D_n(n,S)=D_S(n,S)=\chi(n,S)=1$,
$H(n,S)=\kappa n - \mu n^{\alpha}$
$(\kappa\in \R,\,\mu>0, \, \alpha>1)$ and
$K(n,S)=-S+n$,
the model \eqref{Sys:KS} reads 
\begin{align}	\label{Sys:logistic}					
    \begin{cases}
        n_t = \Delta n - \nabla \cdot (n\nabla S) 
                                             + \kappa n - \mu n^{\alpha},
        \\
       S_t = \Delta S -S + n,
    \end{cases}
\end{align}
and it is usually considered 
in a smoothly bounded domain $\Omega \subset \R^{N}$ 
$(N \in \mathbb{N})$ 
under homogeneous Neumann boundary conditions. 
In \cite{OTYM-2002-NA}, 
global weak solutions
of \eqref{Sys:logistic} with $\alpha=2$ 
were constructed  
for nonnegative initial data in certain spaces
when $N=2$. 
Later on,
it was proved 
in \cite{W-2010-CPDE} that 
there exists a unique global bounded classical solution 
of \eqref{Sys:logistic} with $\alpha=2$ 
in $N$-dimensional bounded convex domains,  
provided that 
$\mu$ is large enough.
Moreover, 
in \cite{C-2017-DCDS-B},   
the convexity 
assumption
was removed 
and it was established that the solution converges
to $(\frac{\kappa_{+}}{\mu}, \frac{\kappa_{+}}{\mu})$ 
in $L^\infty(\Omega) \times L^\infty(\Omega)$  
as $t \to \infty$.
On the other hand, 
according to \cite{L-2015-JDE}, 
it was proved that 
if $\Omega$ is convex 
with $N=3$ 
and $\kappa$ is not too large,  
then weak solutions 
of \eqref{Sys:logistic} 
with $\alpha=2$ 
become smooth 
after some time
for all $\mu > 0$. 
Also, in \cite{VW-2028-DCDS-B}, 
it was shown that 
for $\alpha \in 
(\frac95, 2)$ and $N=3$, 
there exist global very weak solutions 
of \eqref{Sys:logistic}, 
which become smooth 
after some time. 
As for the parabolic--elliptic variant, 
not only global solvability 
but also finite-time blow-up 
were studied 
(see e.g., \cite{TW-2007-CPDE}, \cite{KS-2016-NA}, \cite{W-2018-ZAMP}, \cite{F-2021-NoDEA}). 
There are also some studies 
in the case that 
the function $H$ depends 
on the spatial variable $x$ 
as well 
(see e.g., \cite{F-2020-NARWA}, \cite{BFL-2021-ZAMP}, \cite{BFLM-2023-NARWA}, \cite{FLM-2025-DCDS}).

Apart from this, when $D_n(n,S)=nS$, $\chi(n,S)=n^{\alpha-1} S$ $(\alpha>1)$, 
$H(n,S)=nS$, $D_S(n,S)=1$ 
and
$K(n,S)=-nS$,
the model \eqref{Sys:KS} 
is reduced to the doubly degenerate nutrient taxis system 
\begin{align} \label{DD} 
    \begin{cases}
        n_t = \nabla \cdot (nS \nabla n)
               - \nabla \cdot (n^{\alpha} S \nabla S)
        +nS,
        \\
       S_t = \Delta S - nS.
    \end{cases}
\end{align}
When $N=1$ and $\alpha=2$, 
for all sufficiently regular initial data  
it was shown in \cite{W-2021-TAMS} that 
there exist global weak solutions $(n,S)$ 
of \eqref{DD} 
and that  
these solutions 
converge to $(n_{\infty}, 0)$ 
in $L^\infty(\Omega) \times W^{1,\infty}(\Omega)$ 
with some $n_{\infty} \in C(\overline{\Omega})$ 
as $t \to \infty$, 
provided that 
the initial datum $n_0$  
satisfies 
$\int_\Omega \log n_0(x) \,dx > - \infty$.
This restriction was removed 
in \cite{LW-2022-CPAA}. 
Moreover, 
in \cite{L-2022-JDE}, 
it was shown that 
if $\Omega$ is convex, 
then the system \eqref{DD} 
has a global weak solution, 
provided that 
$\alpha \in (1, \frac32)$ when $N=2$, and 
$\alpha \in (\frac76, \frac{13}{9})$ when $N=3$.
This was improved 
in \cite{W-2024-JDE} 
for $N=2$ 
in the sense that 
the system \eqref{DD} admits 
a bounded global weak solution, 
provided that 
$\alpha \in (1,2)$, or 
$\alpha = 2$ 
with smallness assumptions 
on initial data. 
Recently, 
the smallness assumptions 
on initial data 
and 
convexity assumption on 
the domains
have been removed 
under the condition
$\int_\Omega \log n_0(x) \,dx > - \infty$ 
in \cite{ZL-2026-MMMAS}. 
We refer to 
\cite{W-2024-NNDEA}, \cite{V-2016-JMAA}, \cite{V-2017-NARWA},
\cite{MVY-2025-NA}, \cite{C-2025-JMAA}, \cite{FFV-2025-AAM} 
and \cite{TW-2023-EECT}
for related studies on other chemotaxis systems.

Nevertheless, all these studies are concerned with situations 
in which organisms have no volumes, 
although organisms have nonzero volumes in real-world settings. 
A seminal modeling approach in this direction was pursued in \cite{PH-2002-CAMQ}, where they introduced the system 
\begin{align}\label{VF1}
    \begin{cases}
        u_t = \nabla \cdot (D(u) \nabla u 
                         - h(u,v) \nabla v)
                                    + f(u,v),
        \\
        v_t = \Delta v + g(u,v)
    \end{cases}
\end{align}
with {\it volume-filling effects}, namely,
the functions $D$ and $h$ take the form
\begin{align}\label{Dhq}
    D(u) = q(u) - q'(u)u, \quad h(u,v) = uq(u)\chi(v),
\end{align}
and the function $q$ satisfies $q(1)=0$ and $q(r) \ge 0$ for all $r \in [0, 1)$.
We note that if $q$ is decreasing, then the quantity $D(u)$ in \eqref{Dhq} is positive for all $0\le u \le 1$, and hence the diffusion mechanism in the first equation from \eqref{VF1} remains nondegenerate.
As for the nondegenerate setting, 
in \cite{HP-2001-AAM} 
it was shown that 
the system \eqref{VF1} with 
$D\equiv 1$ and $f \equiv 0$ 
possesses 
a unique global solution 
when   
$h$ satisfies \eqref{Dhq}, and 
$g$ has the form 
$g(u,v) = g_1(u,v)u - g_2(u,v)v$ 
with $g_1 \ge 0$ and $g_2 \ge \eta > 0$.
This was generalized in \cite{W-2004-NA}, 
where 
existence and uniqueness of global classical solutions to \eqref{VF1} 
were proved  
when $D$ fulfills \eqref{Dhq} and 
has
a positive lower bound, and $f$ depends only on $u$.
Moreover, when $g(u,v) = \gamma u - \beta v$ $(\gamma, \beta > 0)$ and $f \equiv 0$,
the Lyapunov functional for \eqref{VF1} was discovered in \cite{W-2006-PRSESA}, and large-time behavior of classical solutions was discussed in the case that $q(u) = 1-u$.
Furthermore, 
under the same assumptions,
it was proved in \cite{JZ-2009-AA} that for all initial data $(u_0, v_0) \in \big(W^{1,p}(\Omega)\big)^2$ with $p > \max\{N, 2\}$,
there exists a stationary solution
$(U,V) \in \big(C^{2+\theta}
(\overline{\Omega})\big)^2$
of \eqref{VF1} 
with some $\theta \in (0,1)$
such that the corresponding global classical solution $(u,v)$ converges to $(U,V)$ in $\big(W^{1,p}(\Omega)\big)^2$ as $t \to \infty$.
We refer to \cite{DR-2008-N, W-2010-NA, WWW-2011-N} for related works on whole spaces, or regarding the singular diffusion, or in a parabolic--elliptic variant.

In contrast to the nondegenerate case,
uniqueness or large-time behavior of solutions to
the degenerate chemotaxis systems with volume-filling effects seems much less 
understood.
In \cite{LW-2005-PNDEA}, they considered the system
\begin{align}\label{VF2}
    \begin{cases}
        u_t = \nabla \cdot (D(u) \nabla u 
                         - h(u) \nabla v),
        \\
        v_t = \Delta v + g(u,v)
    \end{cases}
\end{align}
with $D(1)=0$ and $h(0)=h(1)=0$,
and showed existence and uniqueness of global weak solutions to \eqref{VF2} under appropriate assumptions on $D$, $h$ and $g$.
They moreover constructed flat-hump-shaped and spike-shaped stationary solutions.
After that, the case $D(0) = D(1) = h(0) = h(1) = 0$ was treated in \cite{BKU-2007-MMMAS}, and existence and local H\"{o}lder continuity of global weak solutions were established,
whereas uniqueness of solutions is still unknown.
A similar result was proved in \cite{BBRU-2009-MMAS} for a related problem involving $p$-Laplace type diffusion.
Recently, 
in \cite{SY1}, 
the authors have considered the degenerate 
system 
\begin{align}\label{Sys:SY1}
    \begin{cases}
        u_t = \nabla \cdot (D(u,v) \nabla u 
                         - h(u,v) \nabla v),
        \\
        v_t = \Delta v + g(u,v)
    \end{cases}
\end{align}
with volume-filling effects, 
namely, 
$D(1,\cdot)=0$ and $h(0,\cdot)=h(1,\cdot)=0$, 
and 
similar results as in \cite{LW-2005-PNDEA} have been proved. 
We refer to \cite{W-2010-MMNP, IS-2014-CMA, CIST-2020-DDSSB, SCSS-2025-NARWA} for further discussions.

However, 
to the best of our knowledge,
there is no study on volume-filling chemotaxis
systems with 
source terms.
Thus, the following question naturally arises. 
\begin{center}
{\it
Is it possible to show existence and uniqueness of global weak solutions
\\
when a source term is added to the first equation in \eqref{Sys:SY1}?
}
\end{center}

The purpose of this paper is 
to give an answer to this question by proving
existence and uniqueness of global weak solutions to the degenerate volume-filling 
chemotaxis system with source terms,
\begin{align}\label{Sys:Main} 
    \begin{cases}
        u_t = \nabla \cdot (D(u,v) \nabla u 
                         - h(u,v) \nabla v) + f(x,u,v),
        &  x\in \Omega, \  t>0,
        \\
        v_t = \Delta v + g(u,v),
        &  x\in \Omega, \  t>0,
        \\
        (D(u,v) \nabla u - h(u,v) \nabla v) \cdot \nu = \nabla v \cdot \nu = 0,
        &  x\in \partial\Omega, \  t>0,
        \\
        u(x, 0) = u_0(x), \ v(x, 0) = v_0(x),
        &  x\in \Omega
    \end{cases}
\end{align}
in a smoothly bounded domain $\Omega \subset \R^N$ $(N \in \mathbb{N})$,
where $\nu$ is the outward normal vector to $\pa\Omega$, where
the functions $D$, $h$, $f$ and $g$ 
fulfill the following conditions:
\begin{align} \label{Con}
	\begin{cases}
		D, h, g \in C^2([0,1] \times [0,\infty)), \quad
		f        \in C^2(\overline{\Omega} \times [0,1] \times [0,\infty)),
		\\[1mm]
		D(1,\cdot) = 0, \quad 
		h(0,\cdot) = h(1,\cdot) = 0, \quad
		f(\cdot,0,\cdot) = f(\cdot,1,\cdot) = 0,
		\\[1mm]
		\exists \, D_0 \in C^1([0,1]) \quad \mbox{s.t.} 
         \quad D(r,s) \ge D_0(r) > 0 \quad 
         \forall \, (r,s) \in [0,1) \times [0,\infty), 
         \\[1mm]
        g(\cdot ,0) \ge 0 \quad \mbox{and} 
          \quad \exists\, \kappa > 0 \quad \mbox{s.t.} 
          \quad g_s(r,s) \le \kappa 
          \quad \forall\, (r,s) \in [0,1] \times [0,\infty).
	\end{cases}
\end{align}
Examples of $D$ and $h$ 
satisfying \eqref{Con} 
can be given by 
$D(r,s)=(1-r)^2 (s+1)$ 
and 
$h(r,s)=r (1-r) s$, 
and 
$f(r,s)$ admits
the logistic type functions 
$r(1-r)$ and $|x|r(1-r)s$. 
Also, 
$g(r,s)$ can be taken 
as the production type $-s+r$ 
and 
the consumption type $-rs$. 
As for the initial data,  
the functions $u_0$ and $v_0$ satisfy
\begin{align}
\label{Con:ini} 
     \begin{cases}
          u_0 \in L^{\infty}(\Omega) 
          \quad \mbox{with} \quad 0 \le u_0 \le 1
          \ \mbox{a.e.\ in} \ \Omega, 
          \\
          v_0 \in L^{\infty}(\Omega) \cap H^{2}(\Omega)
          \quad \mbox{with} \quad \nabla v_0 \cdot \nu|_{\pa\Omega} = 0
          \quad \mbox{and} 
          \quad v_0 \ge 0 \ \mbox{a.e.\ in} \ \Omega.
    \end{cases}
\end{align}
Then the first main result reads as follows.
\begin{thm}[Existence of global weak solutions] \label{Thm}
Assume that 
\eqref{Con} and \eqref{Con:ini} hold.
Then one can find functions
\begin{align*} 
    \begin{cases}
          u \in C_{\mathrm{w}}([0,\infty);L^2(\Omega)) \cap
                      L^\infty(\Omega \times (0,\infty))
          \quad \mbox{and} \quad            
          \\
          v \in C([0,\infty);L^2(\Omega)) \cap
                      L^{2}_{\mathrm{loc}}([0,\infty);H^2(\Omega)) \cap
                 L^{\infty}_{\mathrm{loc}}(\overline{\Omega} \times [0,\infty))
    \end{cases}
\end{align*}  
with the properties that
$u\le 1$ a.e.\ in $\overline{\Omega} \times [0, \infty)$, that 
\begin{align*} 
    \begin{cases}
          \mathscr{D}(u,v) \in L^{2}_{\mathrm{loc}}([0,\infty);H^1(\Omega))
          \quad \mbox{and} \quad            
          \\
          \mathscr{D}_s(u,v) \nabla v 
              \in L^{2}_{\mathrm{loc}}([0,\infty);(L^2(\Omega))^N),
    \end{cases}
\end{align*}
and that $(u,v)$ forms a global weak solution of \eqref{Sys:Main}
in the sense of Definition \ref{Def:WS}, 
where 
\begin{align} \label{Def:hanaD}
\mathscr{D}(r,s) := \int_0^r D(\sigma,s) \, d\sigma
\quad \text{for  $r \in [0,1]$ and $s \in [0, \infty)$.}
\end{align}
\end{thm}
In order to give uniqueness 
of global weak solutions to \eqref{Sys:Main}
when 
$D(r,s) = D(r)$
for all $r\in[0,1]$ and $s\in[0,\infty)$ 
we suppose that
\begin{align} \label{SCon}
    \begin{cases}
          \forall\, K > 0 \quad \exists\, C_0, C_1 > 0 \quad \mbox{s.t.} 
               \quad \forall\, (r_1, s_1), (r_2, s_2) \in [0,1] \times [0,K] 
               \\
           \quad    (h(r_1, s_1) - h(r_2, s_2))^2 
               \le C_0 (r_1 - r_2) (\mathscr{D}(r_1) - \mathscr{D}(r_2)) 
                     + C_1 (s_1 - s_2)^2, \\ 
\quad \mbox{where} \ \mathscr{D}(r) := \int_0^r D(\sigma) \, d\sigma \ \mbox{for} \ r \in [0,1],
              \\[1mm] 
          \forall\, K > 0 \quad  \exists\, C_2, C_3 > 0 \quad \mbox{s.t.} 
               \quad \forall \, x \in \overline{\Omega} \quad
               \forall\, (r_1, s_1), (r_2, s_2) \in [0,1] 
                                                         \times [0,K] 
               \\ 
       \quad    (f(x, r_1, s_1) - f(x, r_2, s_2))^2 
               \le C_2 (r_1 - r_2) (\mathscr{D}(r_1) - \mathscr{D}(r_2)) 
                     + C_3 (s_1 - s_2)^2,
              \\[1mm]
         \exists\, C_4 > 0 \quad \exists\, g_1, g_2 \in C^2 ([0,
          \infty)) 
               \ \mbox{with}  \ 
               g_1(0) \ge 0, \  g_2(0) \ge 0 \quad \mbox{s.t.}
               \\
         \quad   g(r,s) = g_1(s) + r g_2(s), 
               \ \, \max \{g_1'(s), g_2'(s)\} \le C_4 
               \quad \forall \, (r,s) \in [0,1] 
                                                 \times [0,\infty).\!\!
     \end{cases}
\end{align}
For example, 
the functions 
$D(r)=(1-r)^2$, 
$h(r,s)=r(1-r)^2 s$ and 
$f(x,r,s)=|x| r(1-r)^2 s$ 
fulfill 
\eqref{SCon}
as well as
\eqref{Con}. 
Also, 
both the production type $g(r,s)=-r+s$ 
and 
consumption type $g(r,s)=-rs$ 
satisfy 
\eqref{SCon}.
Then the second main result reads as follows.
\begin{thm}[Uniqueness of global weak solutions] \label{Thm2}
In addition to Theorem \ref{Thm}, let $D(r,s) = D(r)$
for all $r\in[0,1]$ and $s\in[0,\infty)$, 
$v_0 \in W^{2,p}(\Omega)$ with $p>N$, and 
let \eqref{SCon} hold.
Then
the global weak solution $(u,v)$ of \eqref{Sys:Main} 
in the sense of Definition \ref{Def:WS} 
is unique.
\end{thm}

\begin{remark}\label{rmk}
In \cite{LW-2005-PNDEA} and 
\cite{SY1},
the mass conservation property was required in order to prove uniqueness of global weak solutions, which is not applicable 
to our study.
In contrast, we will follow another argument by choosing a different test function in the weak formulation of the first equation in \eqref{Sys:Main}, so that the proof of Theorem \ref{Thm2} does not necessitate the mass conservation property.
The detailed discussion will be presented
in Section \ref{Sec:GWSU}.
\end{remark}

\noindent
{\bf Key idea of the proof.}
The proof of Theorem \ref{Thm} 
will be proceeded 
by showing that a limit function $(u, v)$ 
of approximate solutions $(u^\ep, v^\ep)$ to \eqref{Sys:Main}
is a global weak solution. 
We need some uniform estimates 
for approximate solutions,  
and
especially, it is important to obtain uniform estimates 
for $f(x, u^\ep, v^\ep)$. 

As mentioned in Remark \ref{rmk}, 
one cannot prove Theorem \ref{Thm2}
in a similar way 
as in previous studies 
(\cite{LW-2005-PNDEA}, \cite{SY1}). 
In our study, 
the mass conservation property
\begin{align} \label{MCP}
\int_\Omega u(x,t) \, dx = \int_\Omega u_0(x) \, dx 
\quad \forall \, t \ge 0,
\end{align}
which was fulfilled in \cite{LW-2005-PNDEA} and \cite{SY1}, 
does not hold 
since there is a source term $f$ in the first equation in \eqref{Sys:Main}. 
In \cite{LW-2005-PNDEA} and \cite{SY1}, 
they set $U:=u-\widehat{u}$ 
for two weak solutions $(u, v)$ and $(\widehat{u}, \widehat{v})$ and 
they chose $\varphi$ 
as a test function of the weak formulation of the first equation, 
where 
$\varphi$ is a unique solution
of the problem
\begin{equation}\label{P'}
\begin{cases}
  -\Delta\varphi = U,  &x\in\Omega, \ t>0, \\
  \nabla\varphi \cdot \nu = 0,   &x\in \partial \Omega, \ t>0
\end{cases}
\end{equation}
with
$\int_{\Omega} \varphi \, dx = 0$.
In order to guarantee 
the existence of $\varphi$, 
we need the mass conservation property \eqref{MCP}
for $u$ and $\widehat{u}$.
However, we cannot rely on \eqref{P'} 
since the mass conservation property 
does not hold in our study. 
Therefore, we turn our eyes to
the problem
\begin{equation*}
\begin{cases}
  -\Delta\varphi +\varphi = U,  &x\in\Omega, \ t>0, \\
  \nabla\varphi \cdot \nu = 0,   &x\in \partial \Omega, \ t>0,
\end{cases}
\end{equation*}
which admits 
a unique solution 
even when 
$u$ and $\widehat{u}$ do not 
satisfy \eqref{MCP}.
Choosing the solution $\varphi$
as a test function in the weak formulation,
we observe that $U=u-\widehat{u}=0$, 
which concludes uniqueness.

\medskip
\noindent
{\bf Organization of this paper.}
The remainder of this paper is organized as follows. 
In Section \ref{Sec:Prelimi}, we estimate 
approximate solutions  
of the problem \eqref{Sys:Main}. 
Section \ref{Sec:GWS} is devoted to the proof 
of Theorem \ref{Thm}.
Finally,
we prove Theorem \ref{Thm2}
in Section \ref{Sec:GWSU}.
\section{Preliminaries} \label{Sec:Prelimi} 
In this section we assume \eqref{Con} and  
\eqref{Con:ini}.
As a first step within our procedure,
we construct approximate solutions 
of \eqref{Sys:Main}.
Fix $\ep \in (0,1)$ and let
\begin{equation} \label{Dep} 
D^\ep(r,s) := D(r,s) + \ep, \quad (r,s) \in [0,1] \times [0,\infty).
\end{equation} 
Also, in view of \eqref{Con:ini}, we can choose 
$(u_0^\ep,v_0^\ep) \in \big(C_{\textrm{c}}^{\infty}(\Omega)\big)^2$ 
such that $0 \le u_0^\ep \le 1$  
and $v_0^\ep \ge 0$ in $\Omega$, 
and that
\begin{equation} \label{111}
\begin{cases}
\| v_0^\ep \|_{L^\infty(\Omega)} \le \| v_0 \|_{L^\infty(\Omega)}+1 
\quad \mbox{and} \quad 
\| v_0^\ep \|_{H^{2}(\Omega)} \le \| v_0 \|_{H^{2}(\Omega)}+1, 
\\
\| u_0^\ep - u_0 \|_{L^2(\Omega)} + \| v_0^\ep - v_0 \|_{L^2(\Omega)} \le \ep.
\end{cases}
\end{equation}
From now on, 
we let $D^{\ep}$, $h$, $f$ and $g$ 
denote 
$C^2$-extensions of themselves.
Then 
there exists a uniquely determined maximal classical solution
\begin{align} \label{kotennkai}
    (u^\ep, v^\ep) \in \big(C(\overline{\Omega} 
                                                  \times [0, T_*)) \cap
                       C^{2,1}(\overline{\Omega} \times (0, T_*))
                       \big)^2
\quad \mbox{with} \ T_* \in (0, \infty]
\end{align}
of the problem
\begin{align}\label{Sys:Mainkinnji} 
    \begin{cases}
        (u^\ep)_t = \nabla \cdot (D^\ep(u^\ep,v^\ep) \nabla u^\ep 
                         - h(u^\ep,v^\ep) \nabla v^\ep) + f(x,u^\ep,v^\ep),
        & x\in \Omega, \  t>0,
        \\
        (v^\ep)_t = \Delta v^\ep + g(u^\ep,v^\ep),
        &  x\in \Omega, \  t>0,
        \\
        (D^\ep(u^\ep,v^\ep) \nabla u^\ep - h(u^\ep,v^\ep) \nabla v^\ep) 
        \cdot \nu = 
        \nabla v^\ep \cdot \nu = 0,
        &  x\in \partial\Omega, \  t>0,
        \\
        u^\ep(x, 0) = u^\ep_0(x), \ v^\ep(x, 0) = v^\ep_0(x), 
        & x\in \Omega
    \end{cases}
\end{align}
by the Amann theory 
(\cite[Section 4, Theorems 14.4 and 14.6]{A-1993-TTM}). 

Let us state some basic properties of 
the unique classical solution $(u^\ep, v^\ep)$
to \eqref{Sys:Mainkinnji}. 
\begin{lem}[Global existence and pointwise estimates] 
\label{Lem:Bounds}
The pair of functions $(u^\ep, v^\ep)$ is a global classical solution to 
\eqref{Sys:Mainkinnji}, that is, $T_*=\infty$ in \eqref{kotennkai}. 
Moreover, for all $T>0$ there exists a positive constant $K(T)$ 
independent of $\ep$ such that 
\begin{align} \label{KK}
0 \le u^\ep(x,t) \le 1 \quad \mbox{and} \quad 0 \le v^\ep(x,t) \le K(T)
\quad \forall \, (x,t)\in \overline{\Omega} \times [0,T].
\end{align}
\end{lem}
\begin{proof}
For functions $w$ we let
$w_+ := \max\{ w, 0 \}$ 
and let $w_- := - \min\{w , 0 \}$.
First, we show that 
\begin{equation} \label{Con:bounds}
    0 \le u^\ep \le 1 \quad \mbox{and} \quad v^\ep \ge 0 \quad 
    \mbox{in} \ \overline{\Omega} \times [0,T_*),
\end{equation}
where $T_*$ is as in \eqref{kotennkai}.
We claim that $u^\ep \ge 0$ in $\overline{\Omega} \times [0,T_*)$. 
To see this, we test the first equation in \eqref{Sys:Mainkinnji} 
by $(u^\ep)_-$
and use the estimate $D^\ep(r,s) \ge \ep$
from \eqref{Dep} to calculate  
\begin{align}
\nonumber
    \frac{d}{dt} \int_\Omega |(u^\ep)_-|^2 \, dx
     &\le  -2\ep \int_\Omega |\nabla (u^\ep)_-|^2 \, dx 
    -2 \int_\Omega h(-(u^\ep)_-,v^\ep) \nabla v^\ep 
                                                          \cdot \nabla(u^\ep)_- \, dx
\\ \nonumber
       &\quad \,  -2 \int_\Omega f(x, -(u^\ep)_- ,v^\ep)\, (u^\ep)_- \, dx 
\\
\nonumber
&\le  -2 \ep \int_\Omega 
            |\nabla (u^\ep)_-|^2 \, dx
        +2 \int_\Omega 
            |h(-(u^\ep)_-,v^\ep) 
                       \nabla v^\ep| 
                |\nabla(u^\ep)_-| \, dx
\\ \label{*}
&\quad \,         
        +2 \int_\Omega 
            |f(x, -(u^\ep)_- ,v^\ep)|
               |(u^\ep)_-| \, dx               
\end{align}
in $(0, T_*)$.
Here, 
we use the Young inequality 
to estimate
\begin{align*}
|h(-(u^\ep)_-,v^\ep) \nabla v^\ep| 
|\nabla(u^\ep)_-| 
\le
\frac{1}{4\ep} 
|h(-(u^\ep)_-,v^\ep) \nabla v^\ep|^2 
+ \ep 
|\nabla(u^\ep)_-|^2,
\end{align*}
which along with \eqref{*} yields 
\begin{align} \nonumber 
    \frac{d}{dt} \int_\Omega |(u^\ep)_-|^2 \, dx
    &\le -2\ep \int_\Omega |\nabla (u^\ep)_-|^2 \, dx 
    + 2\ep \int_\Omega |\nabla (u^\ep)_-|^2 \, dx
    \\ \nonumber 
&\quad \, 
    + \frac{1}{2\ep} 
    \int_\Omega |h(-(u^\ep)_-,v^\ep)|^2 |\nabla v^\ep|^2 \, dx
   +2 \int_\Omega 
            |f(x, -(u^\ep)_- ,v^\ep)|
               |(u^\ep)_-| \, dx 
\\ \nonumber             
&= \frac{1}{2\ep} 
    \int_\Omega |h(-(u^\ep)_-,v^\ep)|^2 |\nabla v^\ep|^2 \, dx 
    +2 \int_\Omega 
            |f(x, -(u^\ep)_- ,v^\ep)|
               |(u^\ep)_-| \, dx
\\ \label{TUIKA1}
&=: \frac{1}{2\ep} I_1 + 2 I_2 
\end{align}
in $(0, T_*)$. 
We first focus on $I_1$. 
Then 
it follows from 
\eqref{Con} 
that 
\begin{align} \label{TUIKA2}
I_1 
= \int_\Omega 
      |h(-(u^\ep)_-,v^\ep)|^2 
               |\nabla v^\ep|^2 \, dx 
= \int_\Omega 
      |h(0,v^\ep) - h(-(u^\ep)_-,v^\ep)|^2 
                   |\nabla v^\ep|^2 \, dx.        
\end{align} 
Also, 
thanks to the mean value theorem, 
there exists 
$\xi_1 \in (-(u^\ep)_{-},0)$ 
such that 
\begin{align} \label{TUIKA3}
h(0,v^\ep) - h(-(u^\ep)_-,v^\ep) 
= h_r(\xi_1,v^\ep) (u^\ep)_{-}. 
\end{align} 
Thus 
in light of 
\eqref{TUIKA2} and \eqref{TUIKA3} 
together with  
the regularity of $h$ and $v^\ep$, 
there exists $c_1(\ep) > 0$ 
such that 
\begin{align} \label{TUIKA4}
I_1 \le c_1(\ep) 
         \int_\Omega |(u^\ep)_-|^2 \, dx. 
\end{align} 
We next estimate $I_2$. 
By \eqref{Con} 
we have $f(\cdot, 0, \cdot) = 0$, 
so that 
we observe that 
\begin{align} \label{TUIKA5}
I_2 
= \int_\Omega 
   |f(x, -(u^\ep)_- ,v^\ep)
            - f(x, 0,v^\ep)|
                  |(u^\ep)_-| \, dx.
\end{align} 
Here, 
again by the mean value theorem, 
there exists 
$\xi_2 \in (-(u^\ep)_{-}, 0)$ 
such that 
\begin{align} \label{TUIKA6}
f(x, -(u^\ep)_- ,v^\ep) - f(x, 0,v^\ep)
=- f_r(x, \xi_2,v^\ep) (u^\ep)_{-}.
\end{align} 
Collecting \eqref{TUIKA5} 
and \eqref{TUIKA6}, 
and considering the regularity 
of $f$ and $v^\ep$, 
we obtain that 
\begin{align} \label{TUIKA7}
I_2 \le c_2(\ep) 
         \int_\Omega |(u^\ep)_-|^2 \, dx.
\end{align}  
Therefore, 
by virtue of 
\eqref{TUIKA1}, \eqref{TUIKA4} 
and \eqref{TUIKA7}, 
we arrive at 
\begin{align} \label{TUIKA8}
\frac{d}{dt} \int_\Omega
                     |(u^\ep)_-|^2 \, dx 
\le 
\Big( 
\frac{c_1(\ep)}{2\ep} + 2c_2(\ep) 
\Big) 
\int_\Omega |(u^\ep)_-|^2 \, dx  
\end{align}
in $(0, T_*)$. 
Let $T \in (0, T_*)$ be arbitrary. 
Then 
owing to \eqref{TUIKA8}
along with nonnegativity 
of $u^\ep_0$, 
it follows from the Gronwall lemma that
$\int_\Omega |(u^\ep(x,t))_-|^2 \, dx = 0$ 
for all $x \in \overline{\Omega}$ and $t \in [0,T)$, 
which leads to
$(u^\ep(x,t))_-=0$ 
for all $x \in \overline{\Omega}$ and $t \in [0,T)$,
and hence 
we have $u^\ep(x,t) \ge 0$ for all $(x,t)\in \overline{\Omega} \times [0,T)$.
Since $T\in(0, T_*)$ is arbitrary, 
we conclude that
$u^\ep \ge 0$ in $\overline{\Omega} \times [0,T_*)$.
Copying the argument with $1-u^\ep$ 
instead of $u^\ep$, 
we see that $u^\ep \le 1$ in $\overline{\Omega} \times [0,T_*)$, 
which entails \eqref{Con:bounds}.
We prove that $v^\ep \ge 0$ in $\overline{\Omega} \times [0, T_*)$. 
Testing the second equation in \eqref{Sys:Mainkinnji} by $(v^\ep)_-$ 
and 
recalling that $g(r,0) \ge 0$ and $g_s(r,s) \le \kappa$ for all $(r,s) \in [0,1] \times [0, \infty)$
(see \eqref{Con}), 
we infer that 
\begin{align*}
    \frac{d}{dt} \int_\Omega |(v^\ep)_-|^2 \, dx 
    &= -2 \int_\Omega |\nabla (v^\ep)_-|^2 \, dx 
         -2 \int_\Omega g(u^\ep, v^\ep) (v^\ep)_- \, dx 
    \\
    &\le -2 \int_\Omega (g(u^\ep,v^\ep)-g(u^\ep,0))(v^\ep)_- \, dx 
    \\
    &\le 2\kappa \int_\Omega |(v^\ep)_-|^2 \, dx
\end{align*}
in $(0, T_*)$.
The Gronwall lemma 
thereby
implies that $\int_\Omega |(v^\ep(x,t))_-|^2 \, dx = 0$, which means
$v^\ep \ge 0$ 
in $\overline{\Omega} \times [0,T_*)$.

Next, we show that $T_* = \infty$. 
From
the second equation in \eqref{Sys:Mainkinnji} and 
the assumption $g_s \le \kappa$  (see \eqref{Con}),
we obtain
\begin{align*}
    (v^\ep)_t(x,t) - \Delta v^\ep(x,t)
    &\le  \kappa v^\ep(x,t) +g(u^\ep(x,t),0)
    \\
    &\le  (\kappa+\| g(\cdot,0) \|_{L^\infty((0,1))})\,(v^\ep(x,t)+1)
\end{align*}
for all $(x,t) \in \Omega \times (0,T_*)$. 
Thus, 
by virtue of
\eqref{111}, we can confirm that
\begin{equation} \label{ini:vupbound}
    v^\ep(x,t) \le (\| v_0^\ep \|_{L^\infty(\Omega)}+1) e^{c_2 t} 
    \le (\| v_0 \|_{L^\infty(\Omega)}+2) e^{c_2 t} \quad
    \forall\,(x,t) \in \overline{\Omega} \times [0,T_*),
\end{equation}
where $c_2 := \kappa+\| g(\cdot,0) \|_{L^\infty((0,1))}$. 
Owing to \cite[Theorem 15.5]{A-1993-TTM}, 
we conclude that $T_* = \infty$. 
The estimate \eqref{KK} thereby results from 
\eqref{Con:bounds} and \eqref{ini:vupbound}.
\end{proof}


To prepare our construction of a solution to \eqref{Sys:Main}, 
let us close this section by deriving some estimates 
for approximate solutions independent of $\ep \in (0,1)$.
\begin{lem}[Estimates for approximate solutions] \label{Con:kinnjikai}
For all $T>0$ there exists a positive constant $C(T)$ 
independent of $\ep$ such that 
\begin{align}
\label{Con1}
&\int_0^T \| v^\ep \|_{L^2(\Omega)}^2 \, dt 
+\int_0^T \| \Delta v^\ep \|_{L^2(\Omega)}^2 \, dt 
+ \int_0^T \| (v^\ep)_t \|_{L^2(\Omega)}^2 \, dt
\le C(T), \\
\label{Con3}
&\int_0^T \|\nabla v^\ep \|_{L^2(\Omega)}^2 \, dt \le C(T), \\
\label{Con2}
&\int_0^T \| D^\ep(u^\ep,v^\ep) \nabla u^\ep \|_{L^2(\Omega)}^2 \, dt \le C(T),
\\
\label{Con4}
&\int_0^T \| \nabla [\mathscr{D}^\ep(u^\ep,v^\ep)] \|_{L^2(\Omega)}^2 \, dt
\le C(T), \\
\label{Con5}
&\int_0^T  \| \nabla [\mathscr{D}_0(u^\ep)] \|_{L^2(\Omega)}^2 \, dt \le 
C(T),  
\\
\label{Con6}
&\int_0^T  \| (u^\ep)_t \|_{(H^1(\Omega))'}^2 \, dt \le C(T),
\end{align}
where for $(r,s) \in [0,1] \times [0, \infty)$,
\begin{align}
\label{hanaDep}
\mathscr{D}^\ep(r,s) 
&:= \int_0^r D^\ep (\sigma,s) \, d\sigma \quad \mbox{and}
\\
\label{hanaD0}
\mathscr{D}_0(r)&:= \int_0^r D_0 (\sigma) \, d\sigma.
\end{align}
\end{lem}
\begin{proof}
We prove \eqref{Con1} and \eqref{Con3}. 
Since 
$\nabla v_0^\ep \cdot \nu|_{\partial\Omega} =0$, 
we can apply
the maximal Sobolev regularity 
for parabolic equations (\cite[Lemma 2.1]{IY-2020-DCDS}, \cite[3.1 Theorem]{HP-1997-CPDE}) to obtain 
\begin{align*}
&\int_0^T \| v^\ep \|_{L^2(\Omega)}^2 \, dt 
+\int_0^T \| \Delta v^\ep \|_{L^2(\Omega)}^2 \, dt 
+ \int_0^T \| (v^\ep)_t \|_{L^2(\Omega)}^2 \, dt
\\
&\le
c_1\left(
\|v_0^\ep\|_{H^2(\Omega)}^2 
+\int_0^T \|g(u^\ep, v^\ep) \|_{L^2(\Omega)}^2 \, dt\right) 
\end{align*}
with some positive constant $c_1$ 
independent of $\ep$.
Since the right-hand side is uniformly bounded 
with respect to $\ep$ due to
\eqref{Con},
\eqref{111}
and \eqref{KK}, we obtain
\eqref{Con1}. 
Also, 
since we see from the 
Schwarz inequality 
and the Young inequality that
\begin{align*}
\|\nabla v^\ep \|_{L^2(\Omega)}^2
 &=(v^\ep, -\Delta v^\ep)_{L^2(\Omega)}  \\
 &\le \| v^\ep \|_{L^2(\Omega)} \|\Delta v^\ep \|_{L^2(\Omega)} \\
 &\le 
 \frac12 \| v^\ep \|_{L^2(\Omega)}^2+
 \frac12\|\Delta v^\ep \|_{L^2(\Omega)}^2,
\end{align*}
the estimate \eqref{Con3}
results from \eqref{Con1}.

We now tackle the proof of \eqref{Con2}.
For $r \in [0,1]$ and $s \in [0, \infty)$ we set
\begin{align}
\label{hanaDepba-}
\widetilde{\mathscr{D}^\ep}(r,s) 
&:= \int_0^r \mathscr{D}^\ep (\sigma,s) \, d\sigma.
\end{align}
We infer from \eqref{hanaDep}, \eqref{Dep} and $\ep \in (0,1)$ that
\begin{align} \label{Ine:hanaDep}
   \mathscr{D}^\ep(r,s) 
    = \mathscr{D}(r,s) + \ep r
    \le \mathscr{D}(r,s) + r
    \quad \forall \, r \in [0,1] \ \forall \, s \in [0, \infty),
\end{align}
where
$\mathscr{D}(r,s) = \int_0^r D(\sigma,s) \, d\sigma$
as in \eqref{Def:hanaD}.
From the first and second equations in \eqref{Sys:Mainkinnji}
together with \eqref{hanaDepba-},
we see that
\begin{align} 
\nonumber
\frac{d}{dt} \int_\Omega \widetilde{\mathscr{D}^\ep}(u^\ep,v^\ep) \, dx
&=\int_\Omega (\widetilde{\mathscr{D}^\ep})_r(u^\ep,v^\ep)
                       \cdot(u^\ep)_t \, dx
    +\int_\Omega (\widetilde{\mathscr{D}^\ep})_s(u^\ep,v^\ep)
                        \cdot(v^\ep)_t \, dx 
\\
\nonumber
&= \int_\Omega 
          \mathscr{D}^\ep(u^\ep,v^\ep) 
          \nabla \cdot (D^\ep(u^\ep,v^\ep) \nabla u^\ep 
                                     - h(u^\ep,v^\ep) \nabla v^\ep) \, dx \\
\nonumber
&\quad \, +\int_\Omega (\widetilde{\mathscr{D}^\ep})_s(u^\ep,v^\ep)
                                           \big(\Delta v^\ep + g(u^\ep,v^\ep)\big) \, dx
                                           \\
                                           \nonumber
&\quad \, +\int_\Omega \mathscr{D}^\ep(u^\ep,v^\ep)f(x,u^\ep,v^\ep) \, dx
\\[2mm]
\label{I123}
&=:\mathcal{I}_1+\mathcal{I}_2+\mathcal{I}_3.
\end{align}
From \cite[(3.18) and (3,19), p.10]{SY1}, we know that
\begin{align}\label{Con:ap1strightcon}
&\mathcal{I}_1
\le -c_2(T)\| D^\ep(u^\ep,v^\ep)\nabla u^\ep \|_{L^2(\Omega)}^2 
+ c_3(T) \| \nabla v^\ep \|_{L^2(\Omega)}^2, \\
\label{??}
&
\mathcal{I}_2 
\le c_4(T)
+c_4(T)\|\Delta v^\ep\|^2_{L^2(\Omega)}
\end{align}
with some positive constants 
$c_2(T)$, $c_3(T)$ and $c_4(T)$ 
independent of $\ep$.
Moreover, in light of \eqref{KK}, the continuity of $f$ and \eqref{Ine:hanaDep}, we obtain 
\begin{align}
\nonumber
\mathcal{I}_3
&\le \int_\Omega (\mathscr{D}(u^\ep,v^\ep)+u^\ep) f(x,u^\ep,v^\ep) \, dx \\
\label{TT}
&\le |\Omega| 
(\|\mathscr{D}\|_{L^\infty((0,1)\times(0,K(T)))} +1)
\|f\|_{L^\infty(\Omega\times(0,1)\times(0,K(T)))}.
\end{align}
Thanks to \eqref{Con1} and \eqref{Con3}, 
we collect \eqref{I123}, \eqref{Con:ap1strightcon}, \eqref{??} and  \eqref{TT},
and integrate over $(0,T)$ to obtain  
\eqref{Con2}. 

We next prove \eqref{Con4} and \eqref{Con5}.
In light of \eqref{Ine:hanaDep}, \eqref{Def:hanaD} 
and \eqref{KK}, we have 
\begin{align}\label{@@@}
|(\mathscr{D}^{\ep})_s (u^{\ep}, v^{\ep})|
= |\mathscr{D}_s (u^{\ep}, v^{\ep})|
\le \int_0^{u^{\ep}} |D_s(\sigma, v^{\ep})| \, d\sigma 
\le \|D_s\|_{L^{\infty}((0,1) \times (0, K(T)))} =: c_5(T).
\end{align}
Thus, \eqref{hanaDep} and \eqref{@@@} lead to
\begin{align*}
| \nabla [\mathscr{D}^\ep(u^\ep,v^\ep)] | 
&\le | D^\ep(u^\ep,v^\ep) \nabla u^\ep |
+ | (\mathscr{D}^\ep)_s(u^\ep,v^\ep) | | \nabla v^\ep | \\
&\le | D^\ep(u^\ep,v^\ep) \nabla u^\ep | + c_5(T) | \nabla v^\ep |.
\end{align*}
This in conjunction with \eqref{Con3} and \eqref{Con2} implies \eqref{Con4}.
Moreover, since $D^\ep(u^\ep,v^\ep) \ge D(u^\ep,v^\ep) \ge D_0(u^\ep)$ by
\eqref{Dep} and \eqref{Con}, we see from \eqref{hanaD0} that
\[
| D^\ep(u^\ep,v^\ep) \nabla u^\ep| \ge | D_0(u^\ep) \nabla u^\ep | 
= | \nabla [\mathscr{D}_0(u^\ep)] |
\] 
and hence, the inequality \eqref{Con2} yields \eqref{Con5}.

We 
finally prove \eqref{Con6}. 
To this end, we fix $\varphi \in H^1(\Omega)$ with 
$\| \varphi \|_{H^1(\Omega)} \le 1$, and estimate the quantity
$|\langle (u^\ep)_t,\varphi \rangle_{(H^1(\Omega))', H^1(\Omega)}|^2 $.
It follows from the first equation in \eqref{Sys:Mainkinnji} that
\begin{align}
\nonumber
&| \langle (u^\ep)_t,\varphi \rangle_{(H^1(\Omega))', H^1(\Omega)} |^2 \\
\nonumber
&= 
| \langle \nabla \cdot (D^\ep(u^\ep,v^\ep) \nabla u^\ep 
- h(u^\ep,v^\ep) \nabla v^\ep) + f(\cdot,u^\ep,v^\ep), \varphi \rangle_{(H^1(\Omega))', H^1(\Omega)} |^2 \\
\nonumber
&\le 2 
| \langle D^\ep(u^\ep,v^\ep) \nabla u^\ep 
- h(u^\ep,v^\ep) \nabla v^\ep, \nabla \varphi \rangle_{(H^1(\Omega))', H^1(\Omega)} |^2 
\\ 
\nonumber
&\quad \,+2
| \langle f(\cdot,u^\ep,v^\ep), \varphi \rangle_{(H^1(\Omega))', H^1(\Omega)} |^2 \\
\nonumber
&=2
| \langle \nabla [\mathscr{D}^\ep(u^\ep,v^\ep)]  
- ((\mathscr{D}^\ep)_s(u^\ep,v^\ep) + h(u^\ep,v^\ep)) \nabla v^\ep,
\nabla \varphi \rangle_{(H^1(\Omega))', H^1(\Omega)} |^2 \\
\nonumber
&\quad \, +2
| \langle f(\cdot,u^\ep,v^\ep), \varphi \rangle_{(H^1(\Omega))', H^1(\Omega)} |^2
\\
\label{J1J2}
&=: \mathcal{J}_1 + \mathcal{J}_2.
\end{align}
For $\mathcal{J}_1$, 
we infer from \eqref{@@@} and \eqref{KK}
that
\begin{align} \nonumber 
\mathcal{J}_1
&\le 2
\Big(\| \nabla [\mathscr{D}^\ep(u^\ep,v^\ep)] \|_{L^2(\Omega)} 
  + (c_5(T) + \| h \|_{L^\infty((0,1) \times (0,K(T)))}) 
  \| \nabla v^\ep \|_{L^2(\Omega)} \Big)^2 \| \nabla \varphi \|_{L^2(\Omega)}^2 \\
\label{JJ11}
&\le 4 \Big(\| \nabla [\mathscr{D}^\ep(u^\ep,v^\ep)] \|_{L^2(\Omega)}^2
+ \big(c_5(T) + \| h \|_{L^\infty((0,1) \times (0,K(T)))}\big)^2
  \| \nabla v^\ep \|_{L^2(\Omega)}^2\Big) \| \nabla \varphi \|_{L^2(\Omega)}^2.
\end{align}
For $\mathcal{J}_2$, by means of  
the H\"{o}lder inequality, the continuity of $f$ and 
\eqref{KK}, we have
\begin{align*}
\mathcal{J}_2 
\le 2
\| f(\cdot,u^\ep,v^\ep) \|^2_{L^2(\Omega)} \| \varphi \|^2_{L^2(\Omega)} 
\le 2 |\Omega| \|f\|^2_{L^\infty(\Omega\times(0,1)\times(0,K(T)))} 
\| \varphi \|^2_{L^2(\Omega)}.
\end{align*}
Combining this and \eqref{JJ11} 
with \eqref{J1J2} shows that 
\begin{align*}
\|  (u^\ep)_t  \|_{(H^1(\Omega))'}^2
&\le 
4 \Big(\| \nabla [\mathscr{D}^\ep(u^\ep,v^\ep)] \|_{L^2(\Omega)}^2
+ \big(c_5(T) + \| h \|_{L^\infty((0,1) \times (0,K(T)))}\big)^2
  \| \nabla v^\ep \|_{L^2(\Omega)}^2\Big) \\
&\quad \,  +
2 |\Omega| \|f\|^2_{L^\infty(\Omega\times(0,1)\times(0,K(T)))}. 
\end{align*}
Integrating this over $(0,T)$ and using \eqref{Con3} 
and \eqref{Con4}, we arrive at \eqref{Con6}.
\end{proof}

\section{Existence of global weak solutions:
Proof of Theorem \ref{Thm}} \label{Sec:GWS}
Let us define a global weak solution of \eqref{Sys:Main}.
\begin{df}[Global weak solutions] \label{Def:WS}
Let $u_0$ and $v_0$ satisfy \eqref{Con:ini}. 
Then a couple $(u,v)$ of nonnegative functions 
with $u\le1$ a.e.\ in $\overline{\Omega} \times [0, \infty)$ such that
\begin{align} \label{Def1}
    \begin{cases}
          u \in C_{\mathrm{w}}([0,\infty);L^2(\Omega)) \cap
                      L^\infty(\Omega \times (0,\infty))
          \quad \mbox{and} \quad            
          \\
          v \in C([0,\infty);L^2(\Omega)) \cap
                      L^{2}_{\mathrm{loc}}([0,\infty);H^2(\Omega)) \cap
                 L^{\infty}_{\mathrm{loc}}(\overline{\Omega} \times [0,\infty))
    \end{cases}
\end{align}
will be called a \textit{global weak solution} of \eqref{Sys:Main} if 
\begin{align} \label{Def2}
    \begin{cases}
          \mathscr{D}(u,v) \in L^{2}_{\mathrm{loc}}([0,\infty);H^1(\Omega))
          \quad \mbox{and} \quad            
          \\
          \mathscr{D}_s(u,v) \nabla v 
              \in L^{2}_{\mathrm{loc}}([0,\infty);(L^2(\Omega))^N),
    \end{cases}
\end{align}
where $\mathscr{D}(r,s) = \int_0^r D(\sigma,s) \, d\sigma$
as in \eqref{Def:hanaD}, and if
\begin{align} \label{Def3}
 \int_0^T \int_\Omega [-u \varphi_t 
            +(D(u,v) \nabla u - h(u,v) \nabla v)
                                             \cdot \nabla \varphi
            +f(x,u,v)\varphi ] \, dxdt                                 
         = \int_\Omega u_0 \varphi(0) \, dx
\end{align}
as well as 
\begin{align} \label{Def4}
  \int_0^T \int_\Omega [-v \varphi_t 
            + \nabla v \cdot \nabla \varphi 
            - g(u,v)\varphi] \, dxdt
         = \int_\Omega v_0 \varphi(0) \, dx
\end{align}
hold for all $T > 0$ and $\varphi \in H^1((0,T);H^1(\Omega))$ 
with $\varphi(T) = 0$,
where 
\begin{align} \nonumber
D(u,v)\nabla u
:=\nabla[\mathscr{D}(u,v)]-\mathscr{D}_s(u,v)\nabla v.
\end{align}
\end{df}

\begin{proof}[\bf Proof of Theorem \ref{Thm}] 
As a first step, we construct a limit function $(u,v)$ of $(u^\ep,v^\ep)$.
We define a function $Q\in C^2([0,1])$ by setting 
\begin{align}\label{Q}
Q(r) :=\int_0^r D_0^2(\sigma) \, d\sigma \quad \mbox{for} \ r \in [0,1],
\end{align}
where $D_0$ is such that
$D(r,s)\ge D_0(r)>0$ 
as in \eqref{Con}.
Fix $T>0$.
We claim that
\begin{equation} \label{Puep}
\left\{ Q(u^\ep) \right\}_{\ep \in (0,1)} \, \mbox{is bounded in} \,
\{ w \in L^2((0,T);H^1(\Omega)) \mid 
w_t \in L^1((0,T);(W^{1,N+1}(\Omega))') \}.
\end{equation}
First, we show that 
$\| Q(u^\ep) \|_{L^2((0,T);H^1(\Omega))} \le c_1(T)$
with some positive constant $c_1(T)$
independent of $\ep$. 
From \eqref{KK}, we have 
$Q(u^\ep) \le Q(1)$, 
which leads to 
$\| Q(u^\ep) \|_{L^2((0,T);L^2(\Omega))} \le c_2(T)$
with some positive constant $c_2(T)$
independent of $\ep$.
Also, the estimates \eqref{KK} and \eqref{Con5} provide  
a positive constant $c_{3}(T)$ independent of $\ep$ such that
\begin{align*}
\int_0^T \| \nabla [Q(u^\ep)] \|_{L^2(\Omega)}^2 \, dt
&\le \| D_0 \|_{L^\infty((0,1))}^2
\int_0^T \| \nabla [\mathscr{D}_0(u^\ep)] \|_{L^2(\Omega)}^2 \, dt \le c_{3}(T).
\end{align*}
Therefore, we see that
\begin{align} \label{Quep1}
\| Q(u^\ep) \|_{L^2((0,T);H^1(\Omega))} \le \sqrt{c_2(T)^2 + c_3(T)}=:c_1(T).
\end{align}
Next, we prove  
\begin{align} \label{Quep2}
\| (Q(u^\ep))_t \|_{L^1((0,T);(W^{1,N+1}(\Omega))')} \le c_{4}(T)
\end{align}
with some positive constant $c_{4}(T)$ 
independent of $\ep$.
Let $\varphi \in W^{1,N+1}(\Omega)$.
In view of \eqref{Q}, 
we can estimate
\begin{align} 
\nonumber
|\langle (Q(u^\ep))_t, \varphi \rangle_{(W^{1,N+1}(\Omega))',W^{1,N+1}(\Omega)}| 
&=| \langle ( D_0^2(u^\ep)(u^\ep)_t, \varphi \rangle_{(W^{1,N+1}(\Omega))',W^{1,N+1}(\Omega)} | \\
\nonumber
&=
| \langle (u^\ep)_t, D_0^2(u^\ep) \varphi \rangle_{(H^1(\Omega))',H^1(\Omega)} | \\ 
\label{11}
&\le
\| (u^\ep)_t \|_{(H^1(\Omega))'} \| D_0^2(u^\ep) \varphi \|_{H^1(\Omega)}
\end{align}
in $(0,T)$.
Here, we estimate $\| D_0^2(u^\ep) \varphi \|_{H^1(\Omega)}$.
Noting that $D_0(u^\ep) \le \| D_0 \|_{L^\infty((0,1))}$, that 
$D_0'(u^\ep) \le \| D_0' \|_{L^\infty((0,1))}$ 
due to \eqref{KK}, and that
$\varphi \in W^{1,N+1}(\Omega) \subset L^\infty(\Omega)$,
we infer that
\begin{align} 
\nonumber
&\| D_0^2(u^\ep) \varphi \|_{H^1(\Omega)} \\
\nonumber
&\le \|D_0^2(u^\ep) \varphi\|_{L^2(\Omega)}  
      + \|2D_0(u^\ep)D_0'(u^\ep) (\nabla u^\ep)\varphi 
           +D_0^2(u^\ep) \nabla \varphi\|_{L^2(\Omega)} \\
\nonumber
&\le \|D_0^2(u^\ep) \varphi\|_{L^2(\Omega)} 
+\|2D_0'(u^\ep) \varphi \nabla [\mathscr{D}_0(u^\ep)] \|_{L^2(\Omega)} 
+\|D_0^2(u^\ep) \nabla \varphi\|_{L^2(\Omega)}  \\
\nonumber
&\le 2\Big(\| D_0 \|_{L^\infty((0,1))}^2\| \varphi \|_{H^1(\Omega)}
+\| D_0' \|_{L^\infty((0,1))} \| \varphi \|_{L^\infty(\Omega)}
    \| \nabla [\mathscr{D}_0(u^\ep)] \|_{L^2(\Omega)} \Big),
\end{align}
which along with the  
embeddings $W^{1,N+1}(\Omega) \hookrightarrow L^\infty(\Omega)$ and $W^{1,N+1}(\Omega) \hookrightarrow H^1(\Omega)$ 
yields 
\[
\| D_0^2(u^\ep) \varphi \|_{H^1(\Omega)} 
\le c_{5}\| \varphi \|_{W^{1,N+1}(\Omega)}
            \big(1+\| \nabla [\mathscr{D}_0(u^\ep)] \|_{L^2(\Omega)}\big) 
\]
with some positive constant $c_{5}$ 
independent of $\ep$. 
Plugging this inequality into \eqref{11}, we can obtain
\[
\|  (Q(u^\ep))_t(t) \|_{(W^{1,N+1}(\Omega))'} 
\le c_{5} \| (u^\ep)_t \|_{(H^1(\Omega))'}
             \big(1+\| \nabla [\mathscr{D}_0(u^\ep)] \|_{L^2(\Omega)}\big).
\]
By integrating this inequality over $(0,T)$, 
the estimates \eqref{Con5} and \eqref{Con6} lead to \eqref{Quep2},
which together with \eqref{Quep1} yields
\eqref{Puep}. 
It follows from \eqref{Puep} and the Aubin--Lions theorem that 
$\left\{ Q(u^\ep) \right\}_{\ep \in (0,1)}$ is
relatively compact in 
$L^2(\Omega \times (0,T))$.
Thus, 
$\left\{u^\ep \right\}_{\ep \in (0,1)}$ is relatively compact in 
$L^2(\Omega \times (0,T))$
since $Q$ is increasing on $[0,1]$.
Also, in light of the Arzel\`{a}--Ascoli theorem, 
$\left\{u^\ep\right\}_{\ep \in (0,1)}$ is relatively compact in
$C([0,T];(H^1(\Omega))')$, and \eqref{Con1} implies that
$\left\{v^\ep \right\}_{\ep \in (0,1)}$ is relatively compact 
in $C([0,T];L^2(\Omega))$.
Since 
$\left\{ u^\ep \right\}_{\ep \in (0,1)}$ is bounded 
in $L^\infty(\Omega \times (0,\infty))$ by 
\eqref{KK} and $\left\{v^\ep \right\}_{\ep \in (0,1)}$ 
is bounded 
in $L^2_{\rm{loc}}([0,\infty); H^2(\Omega))$ due to \eqref{Con1}, 
we find a sequence $\{\ep_k\}_{k\in\mathbb{N}}\subset(0,1)$ with $\ep_k \to 0$ as $k\to \infty$ and functions 
\begin{align*} 
    \begin{cases}
         u \in L^\infty(\Omega \times (0,\infty)) 
\cap C([0,\infty);(H^1(\Omega))') 
          \quad \mbox{and} \quad            
          \\
         v \in L^2_{\rm{loc}}([0,\infty); H^2(\Omega)) 
\cap C([0,\infty);L^2(\Omega)) 
    \end{cases}
\end{align*} 
such that
\begin{align} 
\label{0}
&v^{\ep_k} \to v \quad 
\mbox{weakly in} \ L^2((0,T); H^2(\Omega)),  
\\
\label{shuusoku11}
&(u^{\ep_k},v^{\ep_k}) \to (u,v) \quad \mbox{in} \ 
L^2(\Omega \times (0,T))\times L^2(\Omega \times (0,T)), \\
\label{shuusoku2}
&(u^{\ep_k},v^{\ep_k}) \to (u,v) \quad \mbox{in} \ 
C([0,T];(H^1(\Omega))') \times C([0,T];L^2(\Omega)), 
\\
\label{ae}
&u^{\ep_k} \to u \quad \mbox{and} \quad
v^{\ep_k} \to v \quad \mbox{a.e.\ in} \ \Omega \times (0,\infty)
\end{align}
as $k \to \infty$.
We now verify that the limit couple 
$(u,v)$ is a global weak solution of \eqref{Sys:Main} in the sense of 
Definition \ref{Def:WS}. 

First, we prove
\eqref{Def1}, $0 \le u \le 1$ and $v\ge 0$.
As shown above, we already know that 
$u \in L^\infty(\Omega \times (0,\infty))$ and 
$v \in C([0,\infty);L^2(\Omega)) \cap L^2_{\rm{loc}}([0,\infty); H^2(\Omega))$, 
and hence 
it is sufficient to confirm that 
$u(x,t) \le 1$ for a.a.\ $(x,t) \in \overline{\Omega} \times [0,\infty)$,
that $u \in C_{\mathrm{w}}([0,\infty);L^2(\Omega))$
and that 
$v \in L^{\infty}_{\mathrm{loc}}(\overline{\Omega} \times [0,\infty))$.
From \eqref{KK} and \eqref{ae}, it follows that
$u(x,t)\in [0,1]$  
for a.a.\ $(x,t) \in \overline{\Omega} \times [0,\infty)$, 
and that 
$v(x,t) \in [0,K(T)]$ 
for a.a.\ $(x,t) \in \overline{\Omega} \times [0,T]$.
Moreover,  \eqref{shuusoku2} implies 
$u \in C_{\mathrm{w}}([0,\infty);L^2(\Omega))$, 
which entails \eqref{Def1}.

Next, we prove \eqref{Def2}. 
In view of \eqref{Ine:hanaDep}, \eqref{KK} and \eqref{ae},
we obtain $\mathscr{D}(u, v) \in L^2((0,T);H^1(\Omega))$ 
and
\begin{equation}\label{1}
\nabla[\mathscr{D}^{\ep_k}(u^{\ep_k}, v^{\ep_k})] \to \nabla[\mathscr{D}(u, v)] 
\quad \mbox{weakly in}\  L^2((0,T);(L^2(\Omega))^N)
\end{equation}
as $k \to \infty$.
On the other hand, 
in view of 
\eqref{@@@} and \eqref{0}, we have  
$ \mathscr{D}_s(u,v) \nabla v 
\in L^2((0,T);(L^2(\Omega))^N)$ and 
\begin{equation} \label{@}
(\mathscr{D}^{\ep_k})_s(u^{\ep_k},v^{\ep_k})\nabla v^{\ep_k} \to \mathscr{D}_s(u,v) \nabla v 
\quad \mbox{weakly in}\ L^2((0,T);(L^2(\Omega))^N)
\end{equation}
as $k \to \infty$. 
The said inclusions are indeed \eqref{Def2}.

Next, we show \eqref{Def3} and \eqref{Def4}
for all 
$\varphi \in H^1((0,T);H^1(\Omega))$ with $\varphi(T) = 0$.
Setting
\[
D(u,v)\nabla u
 :=\nabla[\mathscr{D}(u,v)]-\mathscr{D}_s(u,v)\nabla v,
 \]
we see from the inclusion 
$\mathscr{D}(u,v)\in L^2((0,T);H^1(\Omega))$ 
(by \eqref{1}) 
and 
the inclusion 
$ \mathscr{D}_s(u,v) \nabla v \in L^2((0,T);(L^2(\Omega))^N)$ 
(by \eqref{@}) 
that
$D(u,v)\nabla u \in  L^2((0,T);(L^2(\Omega))^N)$. 
A combination of \eqref{1} and \eqref{@} yields
\begin{equation} \label{<}
D^{\ep_k} (u^{\ep_k}, v^{\ep_k})\nabla v^{\ep_k} \to D(u,v)\nabla v \quad 
\mbox{weakly in}\ L^2((0,T);(L^2(\Omega))^N)
\end{equation}
as $k \to \infty$.
On the other hand, 
we know from \eqref{ae} and \eqref{0} that
\begin{equation} \label{kokodemo}
h(u^{\ep_k},v^{\ep_k}) \nabla v^{\ep_k} \to h(u,v) \nabla v \quad 
\mbox{weakly in}\ L^2((0,T);(L^2(\Omega))^N)
\end{equation}
as $k \to \infty$.
Making use of \eqref{ae}
and the continuity of $f$,  we have
$
f(\cdot, u^{\ep_k},v^{\ep_k}) \to f(\cdot,u,v)
$
a.e.\ in $\Omega \times (0,T)$ 
as $k \to \infty$. 
Thus in light of \eqref{KK},
we obtain 
\begin{align} \label{22}
f(\cdot, u^{\ep_k},v^{\ep_k}) \to f(\cdot,u,v) 
\quad
\mbox{weakly$^*$ in} \ L^\infty((0,T); L^\infty(\Omega)) 
\end{align}
as $k \to \infty$. 
Testing the first equation in \eqref{Sys:Mainkinnji} by $\varphi$,
we confirm that
\begin{align*} 
\frac{d}{dt}\int_\Omega u^{\ep_k} \varphi \, dx- \int_\Omega u^{\ep_k} \varphi_t \, dx
&=-\int_\Omega  (D^{\ep_k}(u^{\ep_k},v^{\ep_k}) \nabla u^{\ep_k} 
                         - h(u^{\ep_k},v^{\ep_k}) \nabla v^{\ep_k}) \cdot \nabla \varphi \, dx
\\
&\quad \, + \int_\Omega f(x,u^{\ep_k},v^{\ep_k}) \varphi  \, dxdt.  
\end{align*}
Integrating this identity over $(0,T)$ 
and noting that 
$\varphi(T)=0$, we observe that
\begin{align*} 
&\!\int_0^T\!\! \int_\Omega [-u^{\ep_k} \varphi_t  
+(D^{\ep_k}(u^{\ep_k},v^{\ep_k}) \nabla u^{\ep_k} 
                           - h(u^{\ep_k},v^{\ep_k}) \nabla v^{\ep_k}) \cdot \nabla \varphi 
             + f(x,u^{\ep_k},v^{\ep_k}) \varphi]  \, dxdt
                           \\
&=\int_\Omega u^{\ep_k}_0 \varphi(0) \, dx.
\end{align*} 
Passing to the limit as $k \to \infty$, we arrive at \eqref{Def3}
from \eqref{shuusoku11}, \eqref{<}, \eqref{kokodemo}, \eqref{22} and \eqref{111}. 
Similarly, 
from \eqref{0}, \eqref{shuusoku11} and \eqref{111},
we can verify \eqref{Def4}. 
\end{proof}
\section{Uniqueness of global weak solutions:
Proof of Theorem \ref{Thm2}} \label{Sec:GWSU}

For $w \in L^2(\Omega)$, we let $\mathscr{N}w \in H^2(\Omega)$ denote the unique strong solution
of the problem 
\begin{equation} \label{Def:N}
\begin{cases}
  -\Delta(\mathscr{N}w) +\mathscr{N}w = w,  &x\in\Omega, \\
  \nabla(\mathscr{N}w) \cdot \nu = 0,   &x\in \partial \Omega.
\end{cases}
\end{equation}
\begin{proof}[\bf Proof of Theorem \ref{Thm2}]
Let $(u,v)$ and $(\widehat{u},\widehat{v})$ be global weak solutions 
of \eqref{Sys:Main} in the sense of Definition \ref{Def:WS}. 
Set
\[
U := u-\widehat{u} \quad \mbox{and} \quad
V := v-\widehat{v}.
\]
Let $T>0$ and fix $t \in [0,T]$. 
Then 
it follows from
\cite[p.108, Proposition 2.1 (b)$\Rightarrow$(a)]{Showalter-1997} 
that for all $\varphi \in L^2((0,t);H^1(\Omega))$, 
\begin{align}
\nonumber
&\int_0^t  
\langle U_t, \varphi \rangle
_{(H^1(\Omega))', H^1(\Omega)} \, ds \\
\nonumber
&= -\int_0^t \int_\Omega 
(D(u)\nabla u - D(\widehat{u})\nabla \widehat{u}
-(h(u,v)\nabla v-h(\widehat{u},\widehat{v})\nabla v + h(\widehat{u},\widehat{v})\nabla V))
\cdot \nabla \varphi \, dxds \\
\label{A}
&\quad \, -\int_0^t \int_\Omega (f(x,u,v)-f(x,\widehat{u},\widehat{v}))
\varphi \, dxds,
\end{align}
where $U_t$ is understood 
in the sense of distribution.
Here, we can choose $\varphi = \mathscr{N}U \in L^2((0,t);H^1(\Omega))$. 
To see this, testing the first equation in 
\eqref{Def:N} with $w=U$ by $\mathscr{N}U$ gives
\begin{align*}
\int_\Omega |\nabla (\mathscr{N}U)|^2 \, dx
+\int_\Omega (\mathscr{N}U)^2 \, dx
=\int_\Omega U (\mathscr{N}U) \, dx,
\end{align*}
which by means of the Schwarz inequality 
shows that
\begin{align} \label{NUU}
\| \mathscr{N}U \|_{H^1(\Omega)}
\le 
\| U \|_{L^2(\Omega)}.
\end{align}
Since $0 \le u \le 1$ and $0 \le \widehat{u} \le 1$, it holds that $|U| \le 1$.
This together with \eqref{NUU}
yields 
$\mathscr{N}U \in L^2((0,t);H^1(\Omega))$. 
Therefore, we can substitute $\varphi = \mathscr{N}U$ in \eqref{A} and 
observe that
\begin{align}
\nonumber
&\int_0^t  
\langle U_t, \mathscr{N}U \rangle
_{(H^1(\Omega))', H^1(\Omega)}
\, ds \\
\nonumber
&= -\int_0^t \int_\Omega 
(D(u)\nabla u - D(\widehat{u})\nabla \widehat{u} 
-(h(u,v)\nabla v-h(\widehat{u},\widehat{v})\nabla v + h(\widehat{u},\widehat{v})\nabla V))
\cdot \nabla (\mathscr{N}U) \, dxds \\
\nonumber
&\quad \, -\int_0^t \int_\Omega (f(x,u,v)-f(x,\widehat{u},\widehat{v}))(\mathscr{N}U) \, dxds \\[2mm]
\label{B}
&=: I(t) + J(t).  
\end{align}
First, we estimate $I(t)$. 
Noting that 
$\mathscr{D}(r) = \int_0^r D(\sigma) \, d\sigma$, 
we integrate by parts along with \eqref{Def:N} and 
the Schwarz inequality to obtain 
\begin{align}\nonumber
I(t)
&= \int_0^t \int_\Omega (\mathscr{D}(u)-\mathscr{D}(\widehat{u})) 
\Delta(\mathscr{N}U)
\, dxds \\
\nonumber
&\quad \, + \int_0^t \int_\Omega (h(u,v)-h(\widehat{u},\widehat{v}))\nabla v
\cdot \nabla(\mathscr{N}U) \, dxds \\
\nonumber
&\quad \, + \int_0^t \int_\Omega h(\widehat{u},\widehat{v}) \nabla V
\cdot \nabla(\mathscr{N}U) \, dxds \\
\nonumber
&\le - \int_0^t \int_\Omega (\mathscr{D}(u)-\mathscr{D}(\widehat{u}))U \, dxds 
+ \int_0^t \int_\Omega (\mathscr{D}(u)-\mathscr{D}(\widehat{u}))\mathscr{N}U \, dxds\\
\nonumber
&\quad \, + \int_0^t \| \nabla v\|_{L^\infty(\Omega)}
\| \nabla(\mathscr{N}U) \|_{L^2(\Omega)}
\| h(u,v)-h(\widehat{u},\widehat{v}) \|_{L^2(\Omega)} \, ds \\
\nonumber
&\quad \, + \| h \|_{L^\infty((0,1) \times (0,L(T)))}
\int_0^t \| \nabla(\mathscr{N}U) \|_{L^2(\Omega)}
\| \nabla V \|_{L^2(\Omega)} \, ds \\
\label{tenn}
&=:-I_1(t)+I_2(t)+I_3(t)+I_4(t),
\end{align}
where $L(T) := \max \{ \| v \|_{L^\infty(\Omega \times (0,T))},
\| \widehat{v} \|_{L^\infty(\Omega \times (0,T))} \}$.
We first estimate $I_2(t)$ on the right-hand side.
Let $\delta > 0$,
which will be suitably fixed later. 
In light of the Young inequality, 
we have
\begin{align}
\nonumber
I_2(t)&=\int_0^t \int_\Omega (\mathscr{D}(u)-\mathscr{D}(\widehat{u}))\mathscr{N}U \, dxds
\\
\label{Te}
&\le 
\frac{\delta}{2} \int_0^t \int_\Omega (\mathscr{D}(u)-\mathscr{D}(\widehat{u}))^2 \, dxds
+ \frac{1}{2\delta} \int_0^t \int_\Omega 
(\mathscr{N}U)^2 \, dxds.
\end{align}
Here, 
we claim that
\begin{align}\label{hosoku}
\int_0^t \int_\Omega (\mathscr{D}(u)-\mathscr{D}(\widehat{u}))^2 \, dxds
\le \| D \|_{L^\infty((0,1))}
\int_0^t \int_\Omega(\mathscr{D}(u)-\mathscr{D}(\widehat{u}))U\, dxds.
\end{align}
Indeed, if $r \ge \widehat{r}$, then 
it follows from the definition of $\mathscr{D}$
in \eqref{SCon} that 
\begin{align*}
0 \le \mathscr{D}(r)-\mathscr{D}(\widehat{r})
=\int_{0}^{r}D(\sigma) \, d\sigma - \int_{0}^{\widehat{r}}
D(\sigma) \, d\sigma
=\int_{\widehat{r}}^{r}D(\sigma) \, d\sigma
\le \| D \|_{L^\infty((0,1))}(r-\widehat{r}),
\end{align*}
which yields 
\begin{align*}
0 \le (\mathscr{D}(r)-\mathscr{D}(\widehat{r}))^2
\le \| D \|_{L^\infty((0,1))}(\mathscr{D}(r)-\mathscr{D}(\widehat{r}))(r-\widehat{r}).
\end{align*}
This holds also
in the case that $\widehat{r} \ge r$, 
which guarantees \eqref{hosoku}.
Therefore, we see from \eqref{Te} and \eqref{hosoku} that 
\begin{align}
\nonumber
I_2(t)
&\le \frac{\delta}{2}\| D \|_{L^\infty((0,1))}
\int_0^t \int_\Omega(\mathscr{D}(u)-\mathscr{D}(\widehat{u}))U\, dxds
+ \frac{1}{2\delta} \int_0^t \int_\Omega 
(\mathscr{N}U)^2 \, dxds \\
\label{roku}
&= \frac{\delta}{2}\| D \|_{L^\infty((0,1))} \, I_1(t) + \frac{1}{2\delta} \int_0^t \int_\Omega 
(\mathscr{N}U)^2 \, dxds.
\end{align}
Next, we deal with $I_3(t)$.
The Young inequality allows us to find that 
\begin{align}
\nonumber
I_3(t)&=\int_0^t \| \nabla v\|_{L^\infty(\Omega)}
\| \nabla(\mathscr{N}U) \|_{L^2(\Omega)}
\| h(u,v)-h(\widehat{u},\widehat{v}) \|_{L^2(\Omega)} \, ds \\
\label{u7}
&\le  \frac{\delta}{2} \int_0^t \int_\Omega (h(u,v)-h(\widehat{u},\widehat{v}))^2 \, dxds
+ \frac{1}{2\delta} 
\int_0^t \| \nabla v\|_{L^\infty(\Omega)}^2
\| \nabla(\mathscr{N}U) \|_{L^2(\Omega)}^2 \, ds.
\end{align} 
Here, thanks to the assumption on $h$ in \eqref{SCon}, 
we can control the first term on the 
right-hand side of \eqref{u7} as
\begin{align}
\nonumber
&\frac{\delta}{2} \int_0^t \int_\Omega (h(u,v)-h(\widehat{u},\widehat{v}))^2 \, dxds \\
\label{090}
&\le \frac{\delta}{2} c_1(T)
   \int_0^t \int_\Omega (\mathscr{D}(u)-\mathscr{D}(\widehat{u}))U \, dxds
 + \frac{\delta}{2} 
 c_2(T) \int_0^t \int_\Omega V^2 \, dxds
\end{align} 
with some positive constants $c_1(T)$ and $c_2(T)$. 
Thus, plugging \eqref{090} into \eqref{u7} leads to
\begin{align}
\nonumber
I_3(t) 
&\le \frac{\delta}{2} c_1(T)
   \int_0^t \int_\Omega (\mathscr{D}(u)-\mathscr{D}(\widehat{u}))U \, dxds
 + \frac{\delta}{2} 
 c_2(T) \int_0^t \int_\Omega V^2 \, dxds \\
\nonumber
&\quad \, +\frac{1}{2\delta} 
\int_0^t \| \nabla v\|_{L^\infty(\Omega)}^2
\| \nabla(\mathscr{N}U) \|_{L^2(\Omega)}^2 \, ds \\
&=\frac{\delta}{2} c_1(T)
   I_1(t)
 + \frac{\delta}{2} 
 c_2(T) \int_0^t \int_\Omega V^2 \, dxds 
 \label{PPP}
+\frac{1}{2\delta} 
\int_0^t \| \nabla v\|_{L^\infty(\Omega)}^2
\| \nabla(\mathscr{N}U) \|_{L^2(\Omega)}^2 \, ds.
\end{align} 
On the other hand, $I_4(t)$ can be estimated as  
\begin{align}
\nonumber
I_4(t)&=\| h \|_{L^\infty((0,1) \times (0,L(T)))}
\int_0^t \| \nabla(\mathscr{N}U) \|_{L^2(\Omega)}
\| \nabla V \|_{L^2(\Omega)} \, ds \\
\label{707}
&\le 
\frac{\delta}{2} \int_0^t \| \nabla V \|_{L^2(\Omega)}^2 \, ds
+ \frac{1}{2\delta} \| h \|_{L^\infty((0,1) \times (0,L(T)))}^2
\int_0^t \| \nabla(\mathscr{N}U) \|_{L^2(\Omega)}^2 \, ds.
\end{align}
Collecting \eqref{roku}, \eqref{PPP} and \eqref{707} in \eqref{tenn},
we obtain 
\begin{align} 
\nonumber
I(t) &\le - \int_0^t \int_\Omega (\mathscr{D}(u)-\mathscr{D}(\widehat{u}))U \, dxds 
+ \frac{\delta}{2} \| D \|_{L^\infty((0,1))}
\int_{0}^{t}\int_{\Omega} (\mathscr{D}(u)-\mathscr{D}(\widehat{u}))U \, dxds
\\
\nonumber
&\quad \, 
+ \frac{\delta}{2} c_1(T)
   \int_0^t \int_\Omega (\mathscr{D}(u)-\mathscr{D}(\widehat{u}))U \, dxds
 + \frac{\delta}{2} 
 c_2(T) \int_0^t \int_\Omega V^2 \, dxds \\
 \nonumber
&\quad \, 
+ \frac{1}{2\delta} 
\int_0^t \| \nabla v\|_{L^\infty(\Omega)}^2
\| \nabla(\mathscr{N}U) \|_{L^2(\Omega)}^2 \, ds \\
\label{I(t)}
&\quad \,
+ \frac{\delta}{2} \int_0^t \| \nabla V \|_{L^2(\Omega)}^2 \, ds
+ \frac{1}{2\delta} \| h \|_{L^\infty((0,1) \times (0,L(T)))}^2
\int_0^t \| \nabla(\mathscr{N}U) \|_{L^2(\Omega)}^2 \, ds.
\end{align} 
Next, we estimate $J(t)$. In light of the Young inequality, we have 
\begin{align}
\nonumber
J(t)&=-\int_0^t \int_\Omega (f(x,u,v)-f(x,\widehat{u},\widehat{v}))
(\mathscr{N}U) \, dxds \\
\label{+++}
&\le \frac{\delta}{2} \int_0^t \int_\Omega
|f(x,u,v)-f(x,\widehat{u},\widehat{v})|^2 \, dxds 
+ \frac{1}{2\delta} \int_0^t \int_\Omega (\mathscr{N}U)^2 \, 
dxds.
\end{align}
Here, we make use of the assumption on $f$ in \eqref{SCon} to see that 
\begin{align*}
\nonumber
\frac{\delta}{2} \int_0^t \int_\Omega
|f(x,u,v)-f(x,\widehat{u},\widehat{v})|^2 \, dxds 
&\le
\frac{\delta}{2} c_3(T)
\int_0^t \int_\Omega (\mathscr{D}(u)-\mathscr{D}(\widehat{u}))U \, dxds
\\
&\quad \, + \frac{\delta}{2} c_4(T) \int_0^t \int_\Omega V^2 \, dxds
\end{align*}
with some positive constants $c_3(T)$ and $c_4(T)$.
This together with \eqref{+++} shows that  
\begin{align}\nonumber
J(t)&\le 
\frac{\delta}{2} c_3(T)
\int_0^t \int_\Omega (\mathscr{D}(u)-\mathscr{D}(\widehat{u}))U \, dxds
+ \frac{\delta}{2} c_4(T) \int_0^t \int_\Omega V^2 \, dxds
\\
\label{egao}
&\quad \, 
+ \frac{1}{2\delta} \int_0^t \int_\Omega (\mathscr{N}U)^2 \, 
dxds.
\end{align}
Therefore, adding \eqref{I(t)} and \eqref{egao}  yields
\begin{align}
\nonumber
I(t)+J(t) 
&\le \Big(-1+\frac{\| D \|
_{L^\infty((0,1))}+c_1(T)+c_3(T)}{2}\delta \Big)
\int_0^t \int_\Omega (\mathscr{D}(u)-\mathscr{D}(\widehat{u}))U \, dxds
\\
\nonumber
&\quad \, +\frac{c_2(T)+c_4(T)}{2}\delta
\int_0^t \int_\Omega V^2 \, dxds
+ \frac{\delta}{2} \int_0^t \int_\Omega | \nabla V |^2 \, dxds \\
\nonumber
&\quad \, + \frac{1}{2\delta} \int_0^t 
\big(\| \nabla v 
\|_{L^\infty(\Omega)}^2 + \| h \|_{L^\infty((0,1) \times (0,L(T)))}^2 \big)
\Big(\int_\Omega 
| \nabla(\mathscr{N}U) |^2 \, dx\Big)ds \\
\label{HHH}
&\quad \, + \frac{1}{2\delta} \int_0^t \int_\Omega (\mathscr{N}U)^2 \, 
dxds.
\end{align}
Picking $\delta$ such that 
\begin{align}\label{delta}
0<\delta<\min\left\{1, \frac{2}{\| D \|
_{L^\infty((0,1))}+c_1(T)+c_3(T)}\right\},
\end{align}
we combine \eqref{HHH} with \eqref{B} to find
that
\begin{align}
\nonumber
&\int_0^t \langle U_t, \mathscr{N}U \rangle
_{(H^1(\Omega))', H^1(\Omega)} \, ds
\\
\nonumber
&\le \frac{c_2(T)+c_4(T)}{2}\delta
\int_0^t \int_\Omega V^2 \, dxds
+ \frac{\delta}{2} \int_0^t \int_\Omega | \nabla V |^2 \, dxds \\
\nonumber
&\quad \, + \frac{1}{2\delta} \int_0^t 
\big(\| \nabla v 
\|_{L^\infty(\Omega)}^2 + \| h \|_{L^\infty((0,1) \times (0,L(T)))}^2 \big)
\Big(\int_\Omega 
| \nabla(\mathscr{N}U) |^2 \, dx\Big)ds \\ 
\label{+}
&\quad \, + \frac{1}{\delta} \int_0^t \int_\Omega (\mathscr{N}U)^2 \, 
dxds.
\end{align}
Here, we can rewrite the left-hand side of \eqref{+} as
\begin{align} \label{E}
\int_0^t \langle U_t, 
\mathscr{N}U \rangle
_{(H^1(\Omega))', H^1(\Omega)} \, ds
=\frac12 \int_\Omega |\nabla(\mathscr{N}U)|^2 \, dx
+\frac12 \int_\Omega (\mathscr{N}U)^2 \, dx.
\end{align}
To verify this, first we know that
\begin{align}
\nonumber
\frac{d}{dt}\int_\Omega |\nabla(\mathscr{N}U)|^2 \, dx 
&=\lim_{h \to 0} \frac{1}{h} \Big(
\int_\Omega |\nabla(\mathscr{N}U(t+h))|^2 \, dx
-\int_\Omega |\nabla(\mathscr{N}U(t))|^2 \, dx
                                   \Big)
\\
\nonumber
&=\lim_{h \to 0} 
\int_\Omega \frac{\nabla(\mathscr{N}U(t+h)) \cdot 
(\nabla(\mathscr{N}U(t+h))
- \nabla(\mathscr{N}U(t)))}{h} \, dx 
\\
\nonumber
&\quad \, 
+ \lim_{h \to 0} \int_\Omega \frac{\nabla(\mathscr{N}U(t)) \cdot 
(\nabla(\mathscr{N}U(t+h))
- \nabla(\mathscr{N}U(t)))}{h} \, dx
\\[2mm]
\label{K}
&=:K_1(t) +K_2(t).
\end{align}
First, we rewrite $K_1(t)$.
An integration by parts and \eqref{Def:N} 
entail that 
\begin{align*}
\nonumber
K_1(t)
\nonumber
&=\lim_{h \to 0} 
\Big(\frac{-\Delta(\mathscr{N}(U(t+h)-U(t)))}{h}, \ 
\mathscr{N}U(t+h)
\Big)_{L^2(\Omega)}
\\
&=
\lim_{h \to 0} 
\left\langle\frac{U(t+h)-U(t)}{h}-\frac{\mathscr{N}U(t+h)-\mathscr{N}U(t)}{h},\ 
\mathscr{N}U(t+h)
\right\rangle_{(H^1(\Omega))', H^1(\Omega)}
\\
&=\langle U_t(t)-(\mathscr{N}U)_t(t),\  \mathscr{N}U(t) \rangle
_{(H^1(\Omega))', H^1(\Omega)}
\\
&=
\langle U_t(t), \ \mathscr{N}U(t) \rangle
_{(H^1(\Omega))', H^1(\Omega)}
-\frac12\cdot\frac{d}{dt}\int_\Omega (\mathscr{N}U(t))^2 \, dx.
\end{align*}
Therefore, we obtain 
\begin{align} \label{K1}
K_1(t)
=\langle U_t(t), \ \mathscr{N}U(t) \rangle
_{(H^1(\Omega))', H^1(\Omega)}
-\frac12\cdot\frac{d}{dt}\int_\Omega (\mathscr{N}U(t))^2 \, dx.
\end{align}
A similar argument yields
\begin{align} \label{K2}
K_2(t)
=\langle U_t(t), \ \mathscr{N}U(t) \rangle
_{(H^1(\Omega))', H^1(\Omega)}
-\frac12\cdot\frac{d}{dt}\int_\Omega (\mathscr{N}U(t))^2 \, dx.
\end{align}
Thus, we combine \eqref{K1} and \eqref{K2}  
with \eqref{K} to find that
\begin{align*} 
\int_\Omega U_t (\mathscr{N}U) \, dx
=\frac12 \cdot \frac{d}{dt} \int_\Omega |\nabla(\mathscr{N}U)|^2 \, dx
+\frac12 \cdot \frac{d}{dt} \int_\Omega (\mathscr{N}U)^2 \, dx.
\end{align*}
Integrating this identity over $(0,t)$, we have
\begin{align} 
\nonumber
\int_0^t \int_\Omega U_t (\mathscr{N}U(t)) \, dxds
&=\frac12 \int_\Omega |\nabla(\mathscr{N}U(t))|^2 \, dx
-\frac12 \int_\Omega |\nabla(\mathscr{N}U(0))|^2 \, dx \\
\label{XX}
&\quad 
+\frac12 \int_\Omega (\mathscr{N}U(t))^2 \, dx
-\frac12 \int_\Omega (\mathscr{N}U(0))^2 \, dx.
\end{align}
Here, 
it follows from an integration by parts, \eqref{Def:N} 
and the initial conditions for $u$ and $\widehat{u}$ that
\begin{align*} 
\int_\Omega |\nabla(\mathscr{N}U(0))|^2 \, dx 
+\int_\Omega |\mathscr{N}U(0)|^2 \, dx
=0,
\end{align*}
which together with \eqref{XX} concludes \eqref{E}. 
A combination of \eqref{+} and \eqref{E} thus shows that
\begin{align} 
\nonumber
&\frac12 \int_\Omega |\nabla(\mathscr{N}U)|^2 \, dx
+\frac12 \int_\Omega (\mathscr{N}U)^2 \, dx 
\\ 
\nonumber
&\le 
\frac{c_2(T)+c_4(T)}{2}\delta
\int_0^t \int_\Omega V^2 \, dxds
+ \frac{\delta}{2} \int_0^t \int_\Omega | \nabla V |^2 \, dxds \\
\nonumber
&\quad \, + \frac{1}{2\delta} \int_0^t 
\big(\| \nabla v 
\|_{L^\infty(\Omega)}^2 + \| h \|_{L^\infty((0,1) \times (0,L(T)))}^2 \big)
\Big(\int_\Omega 
| \nabla(\mathscr{N}U) |^2 \, dx\Big)ds \\
\label{AA}
&\quad \, + \frac{1}{\delta} \int_0^t \int_\Omega (\mathscr{N}U)^2 \, 
dxds.
\end{align}
On the other hand, 
we argue similarly for \eqref{Def4}
with the assumption on $g$ in
\eqref{SCon} that
there exists a positive constant $c_5(T)$
such that
\begin{align}
\nonumber
&\| V(t) \|_{L^2(\Omega)}^2
+ \int_0^t \int_\Omega | \nabla V |^2 \, dxds \\
\nonumber
&\le c_{5}(T) \int_0^t \| V \|_{L^2(\Omega)}^2 \, ds
+ c_{5}(T) \int_0^t \big(1+\| \nabla v \|_{L^\infty(\Omega)}^2\big)\Big(\int_\Omega 
| \nabla(\mathscr{N}U) |^2 \, dx\Big)ds \\
\label{BB}
&\quad \, 
+c_5(T)\int_0^t \|\mathscr{N}U\|^2_{L^2(\Omega)}\, ds.
\end{align}
Collecting \eqref{AA} and \eqref{BB}, we have
\begin{align*}
&\int_\Omega |\nabla(\mathscr{N}U)|^2 \, dx
+\int_\Omega (\mathscr{N}U)^2 \, dx 
+\| V(t) \|_{L^2(\Omega)}^2 \\
&\le 
((c_2(T)+c_4(T))\delta+c_5(T))
\int_0^t \int_\Omega V^2 \, dxds
+ (\delta-1)\int_0^t \int_\Omega | \nabla V |^2 \, dxds \\
\nonumber
&\quad \, + \Big(\frac{1}{\delta}+c_{5}(T)\Big) 
\int_0^t 
\big(\| \nabla v 
\|_{L^\infty(\Omega)}^2 + \| h \|_{L^\infty((0,1) \times (0,L(T)))}^2 +1 \big)
\Big(\int_\Omega 
| \nabla(\mathscr{N}U) |^2 \, dx\Big)ds \\
\nonumber
&\quad \, + \Big(\frac{2}{\delta} +c_5(T)\Big)\int_0^t \int_\Omega (\mathscr{N}U)^2 \, 
dxds
+c_{5}(T) \int_0^t \| V \|_{L^2(\Omega)}^2 \, ds.
\end{align*}
Since $\delta<1$ by \eqref{delta}, neglecting the second term on the right-hand side, we obtain
\begin{align}
\nonumber
&\| \nabla(\mathscr{N}U)(t) \|_{L^2(\Omega)}^2
+\| \mathscr{N}U(t) \|_{L^2(\Omega)}^2
+\| V(t) \|_{L^2(\Omega)}^2 \\
\nonumber
&\le 
c_6(T)\int_0^t \big(\| \nabla v \|_{L^\infty(\Omega)}^2 
+ \| h \|_{L^\infty((0,1) \times (0,L(T)))}^2 
+ 1 \big)
\\
\label{GWL}
&\hspace{21mm} \times \big(
\| \nabla(\mathscr{N}U) \|_{L^2(\Omega)}^2
+\| \mathscr{N}U \|_{L^2(\Omega)}^2
+\| V \|_{L^2(\Omega)}^2
\big)
\, ds
\end{align}
with some positive constant $c_6(T)$.
Recalling that 
$u,v \in L^{\infty}(\Omega \times (0,t))$, 
and that  
$v_0 \in W^{2,p}(\Omega)$ with $p>N$ and 
$\nabla v_0 \cdot \nu |_{\partial\Omega}=0$, we see from 
the maximal Sobolev regularity for parabolic equations
(\cite[Lemma 2.1]{IY-2020-DCDS}, \cite[3.1 Theorem]{HP-1997-CPDE}) 
that $v \in L^p((0,t);W^{2,p}(\Omega))$,
and hence, 
the Sobolev embedding 
implies 
$\nabla v \in L^2((0,t);(L^\infty(\Omega))^N)$.
Thus, noting that 
$\nabla \mathscr{N}U(0)= 0$ and $\mathscr{N}U(0) =V(0)=0$,
we infer from
the Gronwall lemma in \eqref{GWL} that 
\begin{align*}
\| \nabla(\mathscr{N}U)(t) \|_{L^2(\Omega)}^2
+\| \mathscr{N}U(t) \|_{L^2(\Omega)}^2
+\| V(t) \|_{L^2(\Omega)}^2
=0
\end{align*}
for all $t \in [0,T]$,
which implies
\begin{gather} \label{L}
\nabla(\mathscr{N}U)(t)=0 \quad \mbox{and} \quad
\mathscr{N}U(t) =0,
\\
\label{Vt}
V(t) =0 
\end{gather}
for all $t \in [0,T]$.
In view of \eqref{Def:N} and \eqref{L}, 
an integration by parts enables us to find that  
for all $t \in [0,T]$, 
\begin{align*} 
\int_\Omega U^2(t) \, dx 
&= \int_\Omega (-\Delta(\mathscr{N}U(t)) + \mathscr{N}U(t)) U(t) \, dx 
\\
&= \int_\Omega (-\Delta(\mathscr{N}U(t))) U(t) \, dx 
\\
&= \int_\Omega \nabla(\mathscr{N}U(t)) \cdot \nabla U(t) \, dx
\\
&=0,
\end{align*}
which yields $U(t)=0$ for all $t \in [0,T]$. 
This along with \eqref{Vt} establishes Theorem~\ref{Thm2}.
\end{proof}


\end{document}